\RequirePackage{fix-cm}
\documentclass[smallextended]{svjour3}       
\smartqed  
\usepackage{graphicx,enumerate}
 \usepackage{mathptmx,multirow}
 \usepackage{amsmath,amssymb}
\newcommand{\set}[1]{\left\{#1\right\}}
\newcommand{\norm}[1]{\left\lVert#1\right\rVert}

\newcommand{\abs}[1]{\left\vert#1\right\vert}
\newcommand{\aabs}[1]{\big\lvert#1\big\rvert}

\newcommand{\eps}{\varepsilon}
\newcommand{\ex}[1]{\mathsf{E}\left[#1\right]}

\newcommand{\wt}{\widetilde}
\DeclareMathOperator{\pr}{\mathsf P}
\newcommand{\R}{\mathbb{R}}

\newcommand{\cF}{\mathcal{F}}
\newcommand{\cB}{\mathcal{B}}

\journalname{Methodology and Computing in Applied Probability}
\begin{document}
\title{Stochastic viability and comparison theorems for mixed stochastic differential  equations}
\titlerunning{Stochastic viability and comparison theorems for mixed SDEs}
\author{Alexander Melnikov \and
        Yuliya Mishura \and
        Georgiy Shevchenko}
\institute{A. Melnikov \at
Department of Mathematical and
Statistical Sciences,
University of Alberta,
632 Central Academic Building,
Edmonton, AB T6G 2G1, Canada\\
\email{melnikov@ualberta.ca}
\and
Yu. Mishura \at
              Kyiv National Taras Shevchenko University, Faculty of Mechanics and Mathematics, Department of Probability, Statistics
and Actuarial Mathematics,
Volodymyrska 64, 01601 Kyiv, Ukraine \\
              Tel.: +380-44-259-03-92\\
              Fax: +380-44-259-03-92\\
\email{myus@univ.kiev.ua}
\and
G. Shevchenko \at
              Kyiv National Taras Shevchenko University, Faculty of Mechanics and Mathematics, Department of Probability, Statistics
and Actuarial Mathematics,
Volodymyrska 64, 01601 Kyiv, Ukraine \\
              Tel.: +380-44-259-03-92\\
              Fax: +380-44-259-03-92\\
\email{zhora@univ.kiev.ua}}

\maketitle

\begin{abstract}
For a mixed stochastic differential equation containing both Wiener process and a H\"older continuous process with exponent $\gamma>1/2$, we prove a stochastic viability theorem. As a consequence, we get a result about positivity of solution and a pathwise comparison theorem. An application to option price estimation is given.
\end{abstract}

\section*{Introduction}
In this paper, we consider a  multidimensional  mixed stochastic differential equation of the form
\begin{equation}\label{mainsde-intro}
X(t) = X_0 + \int_0^t \Big(a(s,X(s)) ds + b(s,X(s)) dW(s) + c(s,X(s))dZ(s)\Big),
\end{equation}
where $W$ is a standard Wiener process, and $Z$ is an adapted process, which is almost surely H\"older continuous with exponent $\gamma>1/2$.

The strongest motivation to study such mixed equations comes from financial modeling. The observations of stock prices processes suggest  that they are not self-similar: on a larger time scale (months or years) these processes are smoother and have a longer memory than on a smaller time scale (hours or days). One reason for this is that the random noise in the market is a sum of a more irregular ``trading'' noise rendering irrationality
of the stock exchange, and a more regular ``fundamental'' noise rendering a current  economical situation.
The first random noise component prevails, especially for illiquid instruments, in a shorter time periods (days and hours). The second one takes some time to propagate and becomes essential in a long run,  clearly exhibiting  a long memory. Such phenomena  can be modeled by a sum of a Wiener process $W$ and a fractional Brownian motion $B^H$ with the Hurst parameter $H>1/2$. The behavior of this process on a smaller scale  is
mainly influenced by independent increments of $W$ and its irregularity, while on the larger scale the long memory of $B^H$ dictates the evolution of the process. As a result, the mixed model describes the stock price behavior in a better way. Let us note that the long memory effect in the financial markets and application of the models involving fractional Brownian motion and mixed Brownian-fractional Brownian motion were studied in many papers (see \cite{BSV}, \cite{Cheridito}, \cite{RamaCont}, \cite{JPS},  \cite{Witate} and references therein).

The existence and uniqueness of a solution to \eqref{mainsde-intro}
was established in \cite{guernual}, \cite{kubilius}, \cite{mish08}   and \cite{mbfbm-limit} under different assumptions, and the most general results are obtained in \cite{guernual} and \cite{mbfbm-limit}. Besides this, in \cite{mbfbm-limit}
a limit theorem for mixed stochastic differential equations was established, and we apply this result here.

In this paper we study a stochastic viability property of the solution to equation \eqref{mainsde-intro} and the  applications of this property. A stochastic process
$X$ is called viable in a non-random set $D$, if  starting at $X_0\in D$, the process stays in $D$ almost surely. For It\^o stochastic differential equations such property was studied, for example,   in \cite{aubin-doss}, \cite{doss}, \cite{doss-lenglart}, \cite{filipovic}, \cite{milian}, \cite{zabczyk}. Recently (see \cite{ciotir-rascanu}), a viability result was proved for stochastic differential equations with fractional Brownian motion. Studying a viability property of solution of equation \eqref{mainsde-intro}, we deduce the path-wise  comparison theorems as an important application. The first comparison theorems for It\^o stochastic differential equations were obtained in \cite{ikeda-wata}, \cite{skorokhod},  \cite{yamada}. In papers \cite{melnikov}  and \cite{milian}   multidimensional pathwise comparison theorems were obtained too (in the first one, for processes with jumps). Here we prove a viability result for \eqref{mainsde-intro} and deduce the results on positivity and pathwise comparison for solutions of such equations.

The paper is organized as follows. Section 1 contains some preliminaries on generalized pathwise stochastic integration. It contains also some results about existence and uniqueness of solution to \eqref{mainsde-intro} as well as a limit theorem for for solutions of such equations.
The main results on viability, positivity and pathwise comparison of solutions to mixed stochastic differential equations are presented in Section~2. Section 3 gives some applications of the results of Section~2 to the comparison of option  prices in the market models with long-range dependence. More precisely, we consider here pure fractional and mixed Brownian-fractional-Brownian Cox-Ross-Ingersoll (mBfBm-CIR) models. The existence and uniqueness result for the corresponding stochastic differential equations is proved. The comparison theorem is applied to obtain an upper bound for the price of a European-type option on the interest rate presented by the mBfBm-CIR model.   We also give a viability and comparison result for It\^o stochastic differential equations with random coefficients. The proofs of these results are very similar to those for equations with non-random coefficients, so we
put them to Appendix, where we also discuss some multidimensional pathwise comparison results.

\section{Preliminaries}

Let $(\Omega,\cF,\mathbb F,\pr)$ be a complete probability space with a
filtration $\mathbb F = \set{\cF_t,t\ge 0}$ satisfying usual assumptions. Let $W=\set{\big(W_1(t),\dots, W_m(t)\big), t\ge 0}$ be a standard $m$-dimensional $\mathbb F$-Wiener process and $Z=\set{\big(Z_1(t),\dots, Z_r(t)\big), t\ge 0}$
be an $r$-dimensional $\mathbb F$-adapted $\gamma$-H\"older continuous process on this probability space.

We consider the following mixed
stochastic differential equation in $\R^d$:
\begin{equation}
\label{main}
\begin{gathered}
X(t) = X_0 + \int_0^t \Big(a(s,X(s)) ds + \sum_{k=1}^m b_k(s,X(s)) dW_k(s) + \sum_{j=1}^r c_j(s,X(s))dZ_j(s)\Big),\\ t\in[0,T],
\end{gathered}
\end{equation}
where coefficients $a,b_k,c_j\colon [0,T]\times \mathbb R^d\to \mathbb R^d$ are jointly continuous. For brevity,  we will use  notation
$$
b(s,X(s)) dW(s) := \sum_{k=1}^m b_k(s,X(s)) dW_k(s)$$ and $$ c(s,X(s))dZ(s) :=\sum_{j=1}^r c_j(s,X(s))dZ_j(s).
$$
We note that in equation \eqref{main}, the integral w.r.t.\ Wiener process is defined as the standard It\^o integral, and the integral w.r.t.\ fBm is pathwise generalized Lebesgue--Stieltjes integral, whose definition is given below.

\subsection{Generalized Lebesgue--Stieltjes integral}

Consider two  continuous functions  $f,g\in C(\R^+)$.
For $\alpha\in (0,1)$ and $0\leq a<b$ define fractional derivatives
\begin{gather*}
\big(D_{a+}^{\alpha}f\big)(x)=\frac{1}{\Gamma(1-\alpha)}\bigg(\frac{f(x)}{(x-a)^\alpha}+\alpha
\int_{a}^x\frac{f(x)-f(u)}{(x-u)^{1+\alpha}}du\bigg)1_{(a,b)}(x),\\
\big(D_{b-}^{1-\alpha}g\big)(x)=\frac{e^{-i\pi
\alpha}}{\Gamma(\alpha)}\bigg(\frac{g(x)-g(b)}{(b-x)^{1-\alpha}}+(1-\alpha)
\int_{x}^b\frac{g(x)-g(u)}{(u-x)^{2-\alpha}}du\bigg)1_{(a,b)}(x).
\end{gather*}
Note that the latter  notation is slightly different from the one given in \cite{samko}. This simplifies the notation below.
If $f,g$ are such that  $D_{a+}^{\alpha}f\in L_p[a,b]$,  $D_{b-}^{1-\alpha}g\in
L_q[a,b]$ for some $p\in (1,1/\alpha)$, $q = p/(p-1)$, we can define the generalized (fractional) Lebesgue-Stieltjes integral as
\begin{equation}\label{general l-s}
\int_a^bf(x)dg(x)=e^{i\pi\alpha}\int_a^b\big(D_{a+}^{\alpha}f\big)(x)\big(D_{b-}^{1-\alpha}g\big)(x)dx.
\end{equation}

As $Z$ is almost surely $\gamma$-H\"older continuous, it is easy to see that for any $\alpha\in (1-\gamma,1/2)$, $j=1,\dots,r$
$$
\Lambda_{t,\alpha}(Z_j) :=
\sup_{0\leq
u< v\leq t}\aabs{\big(D_{v-}^{1-\alpha}Z_j\big)(u)}
<\infty.
$$
Thus, the integral with respect to $Z_j$ can be defined by relation \eqref{general l-s},
and it  admits  the following estimate:
\begin{equation}\label{ineq}
\abs{\int_a^b f(s)dZ_j(s)} \le
C_{\alpha} \Lambda_{b,\alpha}(Z_j) \int_a^b\bigg(\frac{\abs{f(s)}}{(s-a)^{\alpha}}
+\int_a^s \frac{\abs{f(s)-f(z)}}{(s-z)^{\alpha+1}}dz\bigg)ds.
\end{equation}

\subsection{Assumptions}
Throughout the paper, $C$ will denote a generic constant which may change from line to line. To emphasize  the dependence of the constants on
some parameters, we will put hem into subscripts. By $\abs{\cdot}$ we denote both absolute value and the Euclidian norm in $\R^n$,
by $(\cdot,\cdot)$ -- scalar product in $\R^n$, for any $n\geq 1$. Also, throughout the paper $\alpha\in(1-\gamma,1/2)$ will be fixed.

%

We will assume that coefficients of \eqref{main}
satisfy the following hypotheses.
\begin{enumerate}[(M1)]
\item for all $t\in[0,T]$, $x\in\mathbb R^d$, $k=1,\dots,m$, $j=1,\dots,r$
\begin{equation*}
\abs{a(t,x)} + \abs{b_k(t,x)} + \abs{c_j(t,x)}\le C(1+\abs{x});
\end{equation*}
\item for all $t\in[0,T]$, $x,y\in\mathbb R^d$, $i=1,\dots,d$, $k=1,\dots,m$, $j=1,\dots,r$
\begin{equation*}
\begin{gathered}
\abs{a(t,x)-a(t,y)}+\abs{b_k(t,x)-b_k(t,y)} + \abs{c_j(t,x)-c_j(t,y)}\\+\abs{\partial_{x_i} c_j(t,x)-\partial_{x_i} c_j(t,y)} \le C\abs{x-y};\end{gathered}
\end{equation*}
\item for some $\beta\in (1-\gamma,1]$ and all $t,s\in[0,T]$, $x\in\mathbb R^d$,  $i=1,\dots,d$, $k=1,\dots,m$, $j=1,\dots,r$
\begin{equation*}\begin{gathered}
\abs{a(t,x)-a(s,x)}+\abs{b_k(t,x)-b_k(s,x)} + \abs{c_j(t,x)-c_j(s,x)}\\+\abs{\partial_{x_i} c_j(t,x)-\partial_{x_i} c_j(s,x)}\le C\abs{t-s}^\beta.
\end{gathered}\end{equation*}
\end{enumerate}

\subsection{Unique solvability of mixed stochastic differential equation and limit theorem}

In order to formulate the results, we need to introduce the following norm for a vector-valued function $f$
\begin{gather*}
\norm{f}_{\infty,t} = \sup_{s\in [0,t]} \left(\abs{f(s)} + \int_0^s \abs{f(s)-f(z)}(s-z)^{-1-\alpha}dz\right).
\end{gather*}
and a seminorm
\begin{gather*}
\norm{f}_{0,t} = \sup_{0\le u<v<t} \left(\frac{\abs{f(v)-f(u)}}{(v-u)^{1-\alpha}} + \int_u^v \frac{\abs{f(u)-f(z)}}{(z-u)^{2-\alpha}}dz\right).
\end{gather*}

The first result is about existence and uniqueness of solution.
\begin{theorem}\label{thm-solution-mixed}
 Equation \eqref{main} has a solution such that
\begin{equation}\label{solution-cond}
\norm{X}_{\infty,T}<\infty\quad\text{a.s.}
\end{equation}
This solution is unique in the class of processes satisfying \eqref{solution-cond}.
\end{theorem}
Another result we need is a limit theorem. Consider a sequence of equations
\begin{equation*}
X^n(t) = X_0 + \int_0^t \Big(a(s,X^n(s)) ds + b(s,X^n(s)) dW(s) + c(s,X^n(s))dZ^n(s)\Big), t\in[0,T],
\end{equation*}
where $\set{Z_n,n\ge 1}$ is a sequence of almost surely $\gamma$-H\"older continuous processes.
\begin{theorem}\label{thm-main}
Assume that $\norm{Z-Z^n}_{0,T}\to 0$ in probability. Then $X^n(t) \to X(t)$ in probability uniformly in $t$.
\end{theorem}
The proofs of Theorems~\ref{thm-solution-mixed} and \ref{thm-main} are exactly the same as those in one-dimensional case given in \cite{mbfbm-limit}.

\section{Main results}

\subsection{Stochastic viability of solution to mixed stochastic differential equation}
We remind the notion of stochastic viability: for a non-empty set $D\subset\R^d$, process
$X=\set{X(t),t\ge0}$ is called viable in $D$ if $\pr(X(t)\in D, t\ge 0)=1$ when $X(0)\in D$.

Assume that $D$ is smooth, i.e.\ there exists a function $\varphi:\R^d\rightarrow\R$, $\varphi\in C^2(\R^d)$ such that its gradient $\partial_x\varphi(x)\neq 0$ when $\varphi(x)=0$ and
$$
D = \set{x: \varphi(x)\ge 0}.
$$
Denote $\partial D = \set{x: \varphi(x)= 0}$.
\begin{theorem}\label{viabilitymixedthm}
Under assumptions
\begin{enumerate}[\rm (VM1)]
\item for any $t\ge 0$, $x\in D$
$$
\left(\varphi'(x), a(t,x)\right)+ \frac12\sum_{k=1}^m \sum_{i,l=1}^d b_{ki}(t,x)b_{kl}(t,x)\varphi''_{il}(x)\ge 0;
$$
\item for any $t\ge 0$, $x\in D$, $k=1,\dots,m$, $j=1,\dots,r$
$$
\left(\varphi'(x), b_k(t,x)\right)=\left(\varphi'(x), c_j(t,x)\right) = 0.
$$
\end{enumerate}
Then $X$ is viable in $D$.
\end{theorem}
\begin{proof}
Since $\norm{X}_{\infty,T}<\infty$ a.s., the process $X$ is a.s.\ continuous, so it is enough to prove that for all $t\in[0,T]$ \ $\pr(X(t)\in D)=1$.

For $x\in \R^d$, $n\ge 1$, we denote $k_n(x) = \frac{x}{\abs{x}}(\abs{x}\wedge n)$,
$$
Z^n(t) = n \int_{( t-1/n)\vee 0}^t k_n(Z_s) ds
$$
and $k_n(Z) = \set{k_n(Z_t),t\in[0,T]}$.

It was proved in \cite[Lemma 2.1]{mbfbm-limit} that
\begin{equation}
\label{ocenkazn-z}
\norm{Z^n-k_n(Z)}_{0,T}\le C K_\gamma(k_n(Z)) n^{1-\gamma-\alpha},
\end{equation}
where $K_\gamma(g) = \sup_{0\le s<t\le T} \abs{g(t)-g(s)}/(t-s)^{\gamma}$ is the H\"older constant of $g$.  Note that
$$
\norm{Z^n-Z}_{0,T}\le \norm{Z^n-k_n(Z)}_{0,T}+ \norm{Z-k_n(Z)}_{0,T}.
$$
Since $Z$ is a.s.\ continuous, it is bounded, and hence  $\norm{Z-k_n(Z)}_{0,T}\rightarrow0$ a.s., $n\rightarrow \infty$. This relation
together with \eqref{ocenkazn-z} and the inequality
$K_\gamma(k_n(Z))\le  K_{\gamma}(Z)<\infty$ gives
$$
\norm{Z^n-Z}_{0,T}\to 0,n\to \infty
$$
a.s. Defining $X^n$ as a solution to
\begin{equation}\label{approx}
X^n(t) = X_0 + \int_0^t \Big(a(s,X^n(s)) ds + b(s,X^n(s)) dW(s) + c(s,X^n(s))dZ^n(s)\Big), t\in[0,T],
\end{equation}
and applying Theorem~\ref{thm-main}, we have that $X^n(t)\to X(t)$ in probability uniformly in $t\in[0,T]$, $n\to\infty$.

Equation \eqref{approx} can be transformed to the It\^o stochastic differential equation
\begin{equation*}
X^n(t) = X_0 + \int_0^t \Big(a^n(s,X^n(s)) ds + b(s,X^n(s)) dW(s),
\end{equation*}
where
\begin{gather*}
a^n(t,x) = a(t,x)+\sum_{j=1}^r c_j(t,x)\dot Z^n_j(t)\\=a(t,x)+\sum_{j=1}^r c_j(t,x) n\big(k_{nj}(Z(t))-k_{nj}(Z((t-1/n)\vee0))\big),
\end{gather*}
and $k_{nj}(x)$ is the $j$th coordinate of $k_n(x)$.

Since $k_n(Z)$ is bounded, the coefficients of this equation satisfy assumptions (H1)--(H4) in Appendix.
It follows from (VM1) and (VM2) that the conditions of Theorem~\ref{stochviabilityrandcoeffs} are satisfied.
Hence, for all $t\in[0,T]$, we have   $\pr(X^n(t)\in D)=1$. Since $X^n(t)\to X(t)$ as $n\to \infty$ in probability and $D$ is closed, we
get  $\pr(X(t)\in D)=1$, as required.
\end{proof}

\subsection{Positivity and comparison theorem}
As the first application of Theorem~\ref{viabilitymixedthm}, we consider conditions supplying positivity of solution to equation \eqref{main}. This question is of particular interest in financial modeling, where the prices of most assets cannot become negative.
\begin{theorem}
Assume that for some $d'\le d$ and any $i=1,\dots,d'$
\begin{enumerate}[\rm (P1)]
\item $X_{0i}\ge 0$.
\item  for any $x\in \R^d$ such that $x_i= 0$, $x_l\ge 0$, $l=1,\dots,d'$, it holds
$a_i(t,x)\ge 0$ and for any $k=1,\dots,m$, $j=1,\dots,r$ \ $b_{ki}(t,x)= c_{ji}(t,x) = 0$.
\end{enumerate}
Then $\pr(X_i(t)\ge 0, i=1,\dots,d', t\in[0,T])=1$.
\end{theorem}
\begin{proof}
Consider the equation
\begin{equation}\begin{gathered}
\label{ne main}
\wt X(t) = X_0 + \int_0^t \Big(\wt a(s,X(s)) ds + \sum_{k=1}^m \wt b_k(s,X(s)) dW_k(s) + \sum_{j=1}^r \wt c_j(s,X(s))dZ_j(s)\Big),\\ t\in[0,T],
\end{gathered}\end{equation}
where
\begin{gather*}
\wt a(t,x) = a\big(t,\abs{x_1},\dots,\abs{x_{d'}},x_{d'+1},\dots,x_d\big),\\ \wt b(t,x) = b\big(t,\abs{x_1},\dots,\abs{x_{d'}},x_{d'+1},\dots,x_d\big),\\
\wt c(t,x) = c\big(t,\abs{x_1},\dots,\abs{x_{d'}},x_{d'+1},\dots,x_d\big).
\end{gather*}
It follows from Theorem~\ref{viabilitymixedthm} that the unique solution $\wt X$ to equation \eqref{ne main} is viable in each
$D_i = \set{x_i\ge 0}$, $i=1,\dots,d'$.

Consequently, under (P1), we have $\pr(\wt X_i(t)\ge 0, t\in[0,T])=1$ for $i=1,\dots,d'$, or, equivalently,
$\pr(\wt X_i(t)\ge 0, i=1,\dots,d', t\in[0,T])=1$. Hence, $\wt X(t)$ solves \eqref{main}, and, by
uniqueness, we have $\pr(X(t) = \wt X(t),t\in[0,T])=1$, concluding the proof.
\end{proof}

Now we turn to a comparison theorem for solutions of stochastic differential equations.

Let $l\in\set{1,\dots,d}$ be a fixed coordinate number and $X^q$, $q=1,2$, be solutions to mixed equations
\begin{equation}\begin{gathered}
X^q(t) = X^q_0 + \int_0^t \Big(a^q(s,X^q(s)) ds +\sum_{k=1}^m b_k(s,X^q_l(s)) dW_k(s) + \sum_{j=1}^rc_j(s,X^q_l(s))dZ_j(s)\Big),\\ t\in[0,T].
\end{gathered}\end{equation}
The coefficients of both equations are assumed to satisfy hypotheses from Subsection 1.2.

We note that the coefficients $b$ and $c$ are the same for both equations and they depend on the $l$th coordinate only. This
assumption seems restrictive, and we discuss it in Remark~\ref{onedimcomparrem}.

The following result can be deduced from Theorem~\ref{viabilitymixedthm} exactly the same way as
Theorem~\ref{compartheorem} is deduced from Theorem~\ref{stochviabilityrandcoeffs}.
\begin{theorem}\label{compartheorem-mixed}
Assume that
\begin{enumerate}[\rm (CM1)]
\item $X^1_{0l}\le X^2_{0l}$.
\item  for any $x^1,x^2\in \R^d$ such that $x^1_l= x^2_l$ it holds
$a^1(t,x^1)\le  a^2(t,x^2)$.
\end{enumerate}
Then $\pr(X^1_l(t)\le X^2_l(t), t\in[0,T])=1$.
\end{theorem}

\section{Applications}
To apply a pathwise comparison theorem to the mixed financial market model, consider the following one-dimensional pure and mixed analogs of the Cox-Ingersoll-Ross market price model, i.e.\ the stochastic differential equations of the form
\begin{equation}\label{cir-pure}
dX(t)=aX(t)dt + \sigma X(t)^\lambda dB^H(t), t\geq 0, X(0)=X_0>0
\end{equation}
and
\begin{equation}\label{cir-mixed}
dX(t)=aX(t)dt + \sigma X(t)^\lambda(dW(t)+dB^H(t)), t\geq 0, X(0)=X_0>0
\end{equation}
with $1/2\leq \lambda<1$.
The diffusion coefficient $b(s,x)=c(s,x)=\sigma x^\lambda$ in  both equations  is not differentiable at zero and therefore does not satisfy  $(M2)$. Therefore,  we can not apply Theorem 1.1 as well as any other known results concerning the existence and uniqueness of the solution of mixed stochastic differential equations involving fractional Brownian motion,  to  \eqref{cir-pure} and \eqref{cir-mixed}.

Suppose for the moment that equation \eqref{cir-pure} has a solution $X$. Consider the integral $\int_0^t X(s)^\lambda dB^H(s)$. According to the results concerning the pathwise H\"{o}lder properties of the integrals with respect to fBm (for example, see \cite{FeyelPradelle}), this integral is H\"{o}lder continuous in $t$ up to order $H$. It follows that the trajectories of $X$ have the same property. So,  the process $X^\lambda$ has H\"{o}lder continuous trajectories of order up to $H\lambda$. In turn, the integral $\int_0^tX(s)^\lambda dB^H(s)$ exists as a path-wise Riemann-Stieltjes integral if the sum of the H\"{o}lder exponents of the integrand and the integrator exceeds $1$ (see, e.g., \cite{Zahle}), i.e.\ if $H\lambda+H>1$. This heuristical argument explains why we consider equation \eqref{cir-pure} only for the values of index $H\in (\frac{1}{1+\lambda}, 1)$. Similarly, the solution of equation \eqref{cir-mixed} has H\"older continuous trajectories up to order $1/2$, and  the process $X^\lambda$ has H\"older continuous trajectories up to order $\frac{\lambda}{2}$, and hence, it must be $\frac{\lambda}{2}+H>1$, or $H\in(1-\frac{\lambda}{2}, 1)$. In particular, if $\lambda=1/2$, then $H\in(2/3, 1)$ for equation \eqref{cir-pure}, and $H\in(3/4, 1)$ for equation  \eqref{cir-mixed}.

Consider \eqref{cir-mixed} on some  interval $[0,T]$ assuming additionally that the processes $W$ and $B^H$ are independent. According to \cite{Cheridito},  there exists a Wiener process $\widetilde{W}$ with respect to the filtration generated by the sum $W+B^H$, such that \begin{equation}
\label{eq9} W(t)+B^H(t) =\widetilde{W}(t) -\int_0^t
{\int_0^s {r \left( {s,u} \right)d\widetilde{W}(u) ds} },
\end{equation}
where the square integrable kernel $\set{r(t,s),0\le s< t\le T}$ is the unique solution of the equation
\begin{equation}
\label{eq8} r\left( {t,s} \right)+\int_0^s
{r\left( {t,x} \right)r\left( {s,x}
\right)dx={ H}\left( {2{ H}-1} \right)
(t-s)^{2{ H}-2}} , \quad 0\le s< t\le T.
\end{equation}
In this case, equation \eqref{cir-mixed} can be reduced to the stochastic differential equation
\begin{equation}\label{cir-mixed-trans}\begin{gathered}
dX(t)=\left(aX(t)-\sigma {X(t)^{\lambda}}\int_0^t{r \left( {t,u} \right)d{\widetilde{W}}(u)}\right)dt + \sigma{X(t)^{\lambda}}d\widetilde{W}(t),\\ t\geq 0, X(0)=X_0>0.
\end{gathered}\end{equation}
Equation \eqref{cir-mixed-trans} does not involve a fractional Brownian motion but it has a random non-Lipschitz and non-Markov drift coefficient.
\begin{theorem}\label{ex-uniq}
\begin{enumerate}[\rm (I)]
\item Equation  \eqref{cir-pure}  has a unique solution if $H\in (\frac{1}{1+\lambda}, 1)$.
\item Equation  \eqref{cir-mixed}  has a unique solution if $H\in (1-\frac{\lambda}{2}, 1)$.
\item Denote $\nu_0=\inf\{t>0: X(t)=0\}$, where $X$ is a solution to either  \eqref{cir-pure} or  \eqref{cir-mixed}.
Then $X(t)= 0$ a.s.\ for all $t\geq \nu_0$.
\end{enumerate}
\end{theorem}
\begin{proof}
The proof of this theorem is given in Appendix C.
\end{proof}

Consider the one-dimensional  mixed  Cox-Ingersoll-Ross interest rate  model described by equation \eqref{cir-mixed}.  We suppose that the discounting factor   equals $B(t)=e^{at}, t\geq 0$. It is impossible to solve explicitly equation \eqref{cir-mixed} and the distribution of its solution has a complicated form. We note that even for the classical Cox-Ingersoll-Ross interest rate model, involving only Wiener process,  it has a noncentral chi-square distribution according to \cite{CIR} (see also a discussion in \cite{ioffe}). To avoid such difficulties, we apply Theorems~\ref{ex-uniq} and \ref{compartheorem-mixed} to get an upper bound for option prices on interest rate in the model described by equation~\eqref{cir-mixed}. To this end, introduce   
$Y(t)=X(t)^{1-\lambda}, t\in[0,\nu_0)$. According to the It\^{o} formula, $Y$ satisfies the following stochastic differential equation on $[0,\nu_0)$:
\begin{equation}\label{dYt}
dY(t)=\Big(a(1-\lambda)Y(t)-\frac{\sigma^2\lambda(1-\lambda)}{2 Y(t)}
\Big)dt+\sigma(1-\lambda)\big(dW(t)+dB^H(t)\big).
 \end{equation}
Consider an auxiliary process $Z$ that satisfies the following stochastic differential equation (Vasi\'{c}ek model)
 \begin{equation}\label{auxiliary}\begin{gathered}
 dZ(t)=a(1-\lambda)Z(t)dt+\sigma(1-\lambda)\big(dW(t)+dB^H(t)\big).
\end{gathered}\end{equation}
with the initial condition $Z(0) = X_0^{1-\lambda}$.
Equation \eqref{auxiliary} has a unique solution on $\R_+$, and it  is a Gaussian process of the  form
\begin{gather*}
Z(t)= X_0^{1-\lambda}
e^{a(1-\lambda)t}+\sigma(1-\lambda)\int_0^te^{a(1-\lambda)(t-s)}\big(dW(s)+dB^H(s)\big).
\end{gather*}
Applying Theorem \ref{compartheorem-mixed} to $Y$ and $Z$ on the interval $[0,\nu_0)$, we get $Y(t)\leq Z(t)$ a.s. for all $t\in[0,\nu_0)$. Since $Y(t)=0$ for $t\ge \nu_0$, this implies  
$Y(t) \leq Z(t)_+$ for all $t\ge 0$.  Therefore, $X(t)\leq Z(t)_+^{1/(1-\lambda)}$ a.s.\ for $t\geq 0$. Considering a European option with a nondecreasing payoff function $f(\cdot)$ and maturity $T>0$, we can use the latter inequality to derive the following upper price bound:
\begin{gather*}
\ex{e^{-aT}f(X_T)}\leq \ex{e^{-aT}f\big(Z(T)_+^{{1}/(1-\lambda)}\big)}
\end{gather*}
where the distribution of $Z(T)$ is  Gaussian with mean ${X_0^{1-\lambda}}e^{a(1-\lambda)T}$ and variance \begin{gather*}
\sigma^2(1-\lambda)^2\bigg(\frac{e^{2a(1-\lambda)T}-1}{2a(1-\lambda)}+H(2H-1)\int_0^T\int_0^T e^{a(1-\lambda)(t+s)}|t-s|^{2H-2}dt\,ds\bigg).
\end{gather*}

For $a=0.1, X_0=1, \lambda=0.5, T = 10, f(x) = (x-K)_+$ we have calculated 
the upper bound for different $\sigma$ and $K$. We have also computed numerically the prices by using the Euler method with $4096$ points to simulate $20000$ paths of the solution. Results are following.

\begin{table}[h]\begin{center}
\begin{tabular}{|c|c|c|c|c|c|c|c|c|c|}
\hline
$\sigma$ &\multicolumn{3}{|c|}{$0.1$}&\multicolumn{3}{|c|}{$0.5$}&\multicolumn{3}{|c|}{$1$}\\
\hline
$K$ & 0.5 & 1 & 2& 0.5 & 1 & 2& 0.5 & 1 & 2\\
\hline
Price &  0.8818  &  0.7015 &  0.4032&  2.17  &  2.04 & 1.891 & 4.448 & 4.42& 4.312 \\
\hline
Upper bound &  0.8953  &  0.7198 & 0.4192 &  2.55 &  2.434  & 2.223 & 6.552 & 6.448 & 6.25  \\
\hline
\end{tabular}
\end{center}
\end{table}

As we see, the difference between the upper bound and the price is minor for small values of $\sigma$ and becomes substantial when $\sigma$ increases. There 
are two reasons for that: first, the term we have omitted in \eqref{dYt} to obtain the upper bound grows with $\sigma$, second, for large $\sigma$ the solution hits zero quickly and stays there, while the upper bound becomes larger. 

It is of great interest whether or not solutions of equations \eqref{cir-mixed} and \eqref{cir-pure} stay positive. For equation \eqref{cir-mixed} with $H>3/4$ a complete answer can be given. Namely, in this case the sum $W+B^H$ under equivalent measure trasformation becomes a standard Wiener process, so equation \eqref{cir-mixed} transforms to an It\^{o} equation. Its solution is a diffusion process, 
so it can be checked by using the results of \cite[Section VI.3]{iwbook} that the solution almost surely vanishes within a finite time. 

For equation \eqref{cir-pure}, the precise answer is not known. The numerical experiments below suggest that for $a<0$ the solution vanishes almost surely, and for $a>0$ it stays positive at least with positive probability. (The Euler method we have used for simulations cannot give and answer whether the probability is $1$ or less.)

We have considered equations
\begin{equation*}
dX_t = a X_t dt + \sqrt{X_t}dB^H_t,
 X_0 = 1,t\ge 0,
 \end{equation*}
with $ H = 0.8$ and $a=\pm 1$.
For $a=0.1$, the following figures contain histogram of the time when trajectories hit $0$ on the segments $[0,50]$ and $[0,500]$.

\noindent\includegraphics[width=6cm]{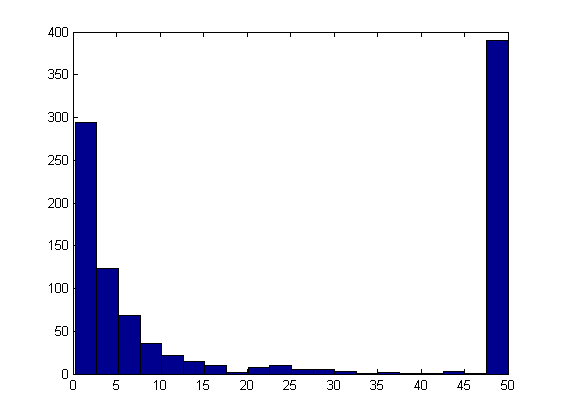}
\includegraphics[width=6cm]{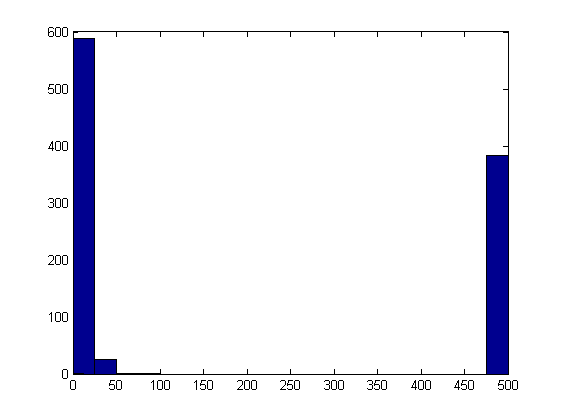}

We see that almost the same number of paths (around $400$) stays positive at $t=50$ and at $t=500$. 

For $a=-0.1$, the following figures contain the graph of the 1000 paths and histogram of the time when trajectories hit $0$.

\noindent\includegraphics[width=6cm]{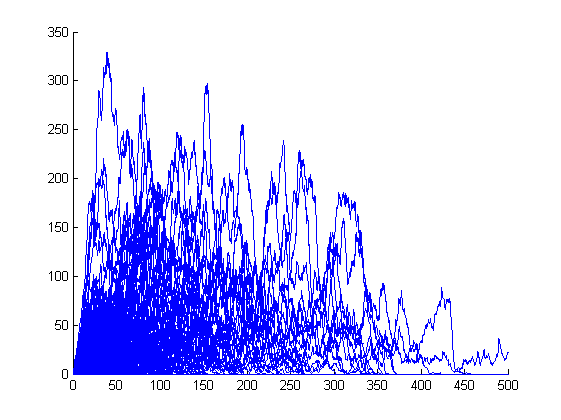}
\includegraphics[width=6cm]{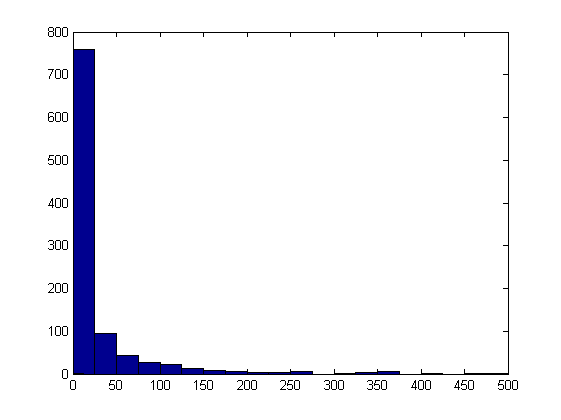}

We see that only a signle path stays positive at $t=500$.

\appendix
\section{Stochastic viability for ordinary stochastic differential equations with random coefficients}

Consider a stochastic differential equation
\begin{equation}
\label{sderandcoeff}
X(t) = X_0 + \int_0^t \Big(a(s,X(s),\omega)ds + \sum_{k=1}^m b_k(s,X(s),\omega)dW_k(s)\Big),\quad t\in\R_+
\end{equation}
where the coefficients $a$ and $b_k$, $k=1,\dots,m$, are measurable functions from $\R_+\times \R^d\times\Omega$
to $\R^d$; $W_k$, $k=1,\dots,m$ are independent $\mathbb F$-Wiener processes, $X_0$ is $\mathcal F_0$-measurable random
vector in $\R^d$. We assume for simplicity that $X_0$ is bounded.

The following assumptions guarantee that there exists a unique solution to \eqref{sderandcoeff} in the class of
adapted square integrable processes (see \cite{skorokhod}).
\begin{enumerate}[(H1)]
\item For each $t\in\R_+$ the functions $a(t,\cdot)$ and $b_k(t,\cdot)$, $k=1,\dots,m$, are $\cB(\R^d)\otimes \cF_t$-measurable.
\item Functions $a(\cdot,\omega)$ and $b_k(\cdot,\omega)$, $k=1,\dots,m$, are almost surely jointly continuous on $\R_+\times \R^d$.
\item For any $t\in\R_+$ and any $x\in\R$  $$\abs{a(t,0)} + \sum_{k=1}^m\abs{b_k(t,0)}\le C.$$
\item For any $t\in\R_+$,
 $x,y\in\R^d$ $$\abs{a(t,x)-a(t,y)} + \sum_{k=1}^m\abs{b_k(t,x)-b_k(t,y)}\le C\abs{x-y}.$$
\end{enumerate}

Let $\varphi:\R^d\rightarrow\R$,  $\varphi\in C^2(\R^d)$ be a function such that its gradient $\partial_x \varphi(x)\neq 0$ when $\varphi(x)=0$. Let the set
$$
D = \set{x: \varphi(x)\ge 0}
$$
be non-empty, and $\partial D = \set{x: \varphi(x)= 0}$ be its boundary.

\begin{theorem}\label{stochviabilityrandcoeffs}
Assume that $X_0\in D$ a.s.\ and the following conditions hold almost surely:
\begin{enumerate}[\rm (V1)]
\item for any $t\in\R_+$, $x\in D$
$$
\alpha(t,x):= \left(\varphi'(x), a(t,x)\right)+ \frac12\sum_{k=1}^m \sum_{i,l=1}^d b_{ki}(t,x)b_{kl}(t,x)\varphi''_{il}(x)\ge 0;
$$
\item for any $t\in\R_+$, $x\in D$, $k=1,\dots,d$
$$
\beta_k(t,x):= \left(\varphi'(x), b_k(t,x)\right)= 0.
$$
\end{enumerate}
Then $\pr(X(t)\in D\text{ for all }t\in\R_+)=1$.
\end{theorem}
\begin{proof}
Since $D$ is closed and $X(t)$ is a.s.\ continuous, it is enough to prove that $\pr(X(t)\in D)=1$ for all $t\ge 0$.

\emph{Step 1}. First we prove that if  $\alpha(t,x) >0$ a.s.\ for all $t\in\R_+$, $x\in D$ and if $X_0\in\partial D$, then there exists a stopping time $\theta>0$ a.s.\ such that $X(t)\in D$ for all $t\in [0,\theta]$.
To this end, fix some $R>\abs{X_0}$ and define stopping times
$$
\tau = \inf\set{s\in \R_+: \alpha(s,X(s))< 0},\qquad \tau_R = \min\set{s\in \R_+: \abs{X(s)}\ge R}.
$$
As usual, we suppose that a stopping time equals $\infty$ if the corresponding set is empty. It is clear that $\tau>0$ and $\tau_R>0$ a.s.

For any $u\ge 0$ put $\theta_u = u\wedge \tau\wedge \tau_R$ and apply the It\^o formula to the process  $\varphi(X(\cdot))$:
\begin{gather*}
\varphi(X(\theta_u)) = \int_0^{\theta_u}\Big(\alpha(s,X(s))ds + \sum_{k=1}^m\beta_k(s,X(s))dW_k(s) \Big).
\end{gather*}
Since $X$ is bounded on $[0, \theta_u]$, the expectation of the stochastic integral equals zero. Hence
\begin{gather*}
\ex{\varphi(X(t))} = \ex{\int_0^{\theta_u} \alpha(s,X(s))ds}.
\end{gather*}

For a non-negative function $\psi\in C(\R)$ such that $\int_\R \psi(x) dx = 1$ and $\psi(x)=0$, $x\notin[0,1]$, define
$$
\psi_n(x) = n\int_0^{\abs{x}} \int_0^y \psi(nz)dz\,dy.
$$
Obviously, $\psi_n(x)\uparrow \abs{x}$ as  $n\to\infty$ and $|\psi'_n(x)|\le 1$, $n\geq 1$.

Applying the It\^o formula to $\psi_n(\varphi(X(\cdot)))$, we get
\begin{gather*}
\psi_n(\varphi(X(\theta_u))) = \int_0^{\theta_u} \psi'_n(\varphi(X(s)))\big(\alpha(s,X(s))ds + \sum_{k=1}^m\beta_k(s,X(s))dW_k(s) \big)\\ + \frac12\int_0^{\theta_u} \psi''_n(\varphi(X(s))) \sum_{k=1}^m \beta_k(s,X(s))^2 ds.
\end{gather*}
Similarly,
\begin{equation}\label{murmur}
\begin{gathered}
\ex{\psi_n(\varphi(X(\theta_u)))} = \ex{\int_0^{\theta_u} \psi'_n(\varphi(X(s)))\alpha(s,X(s))ds} \\+\frac12\sum_{k=1}^m\ex{\int_0^{\theta_u} \psi''_n(\varphi(X(s))) \beta_k(s,X(s))^2}ds.
\end{gathered}\end{equation}
Recall that $\alpha(s,X_s)\ge 0$ for $s<\theta_u$, and $|\psi_n'(x)|\le 1$,
so the first term in the right-hand side of \eqref{murmur} does not exceed $\ex{\int_0^{\theta_u}\alpha(s,X(s))ds}$. We will
 prove now  that the second term vanishes.

For $x\in \R^d$, let $z(x)$ be the closest to $x$ point
such that $\varphi(z(x))=0$ (any of the points if there are more than one). As the function $\varphi$ is continuously differentiable and its derivative is non-zero on $\partial D$,  there exists $\eps>0$ such that
$$
\abs{x-z(x)}\le C_R\abs{\varphi(x)-\varphi(z(x))}=C_R \abs{\varphi(x)}
$$
whenever $\abs{x}\le R$, $\abs{\varphi(x)}<\eps$. Further, for $\abs{x}\le R$
\begin{gather*}
\abs{\beta_k(s,x)} = \abs{\beta_k(s,x)-\beta_k(s,z(x))}\\\le \abs{\big(\varphi'(x)-\varphi'(z(x)),b_k(s,x)\big)} + \abs{\big(\varphi'(z(x)),b_k(s,x)-b_k(s,z(x))\big)}\\\le C_R\big(\abs{x-z(x)} +  \abs{x-z(x)}\big)\le C_R\abs{x-z(x)}.
\end{gather*}
But for $n>1/\eps$ it holds that $$\psi''_n(\varphi(X(s)) = n\psi(n\varphi(X(s)))=0$$ whenever $\abs{\varphi(X(s))}>\eps$,
whence
\begin{gather*}
\ex{\int_0^{\theta_u} \psi''_n(\varphi(X(s))) \beta_k(s,X(s))^2}ds \le C_R
n\int_0^{\theta_u} \psi(n\varphi(X(s)) \varphi(X(s))^2ds\\\le \frac{C_R u}{n}  \sup_{x\in \R} x^2 \psi(x)\to 0,\quad n\to\infty.
\end{gather*}
Letting $n\to\infty$, we get from \eqref{murmur} that
\begin{gather*}
\ex{\abs{\varphi(X(\theta_u))}} \le\liminf_{n\rightarrow\infty}\ex{\psi_n(\varphi(X(\theta_u)))}\le \ex{\int_0^{\theta_u} \alpha(s,X(s))ds} =\ex{\varphi(X(\theta_u))},
\end{gather*}
and hence, $\varphi(X(\theta_u))\ge 0$ a.s. Since $\theta_u=u\wedge \tau\wedge \tau_R$ and $u$ is arbitrary, we get the desired claim with
$\theta = \tau\wedge \tau_R>0$ a.s.

\emph{Step 2}. We prove the statement of the theorem  under assumption that $\alpha(t,x) >0$ a.s.\ for all $t\ge 0$, $x\in D$.

Define $\tau_D = \inf\set{s\ge 0: X_s\notin D}$ and assume on the contrary that $\pr(\tau_D<\infty)>0$. Then for some
$R>0$ it holds $\pr(\tau_D<\infty, |X(\tau_D)|\le R)>0$. Consider a
new stochastic basis $(\wt \Omega, \wt \cF, \wt{\mathbb F}, \wt \pr)$, where
$\wt \Omega = \set{\tau<\infty, |X(\tau_D)|\le R}$; $\wt \cF= \cF|_{\wt\Omega}$;  $\wt{\mathbb F} = \set{\cF_{t+\tau_D}|_{\wt\Omega},t\ge 0}$,  where
$\cF_{t+\tau_D}$ is the $\sigma$-algebra generated by the stopping time $t+\tau_D$; $\wt \pr(A) = \pr(A)/\pr(\Omega)$,
$A\in\wt\cF$. Define also $\wt X(t) = X(t+\tau_D)$, $\wt a(t,x) = a(t+\tau_D,x)$, $\wt b_k(t,x) = b_k(t+\tau_D,x)$, $k=1,\dots,m$, $\wt W(t) = W(t+\tau_D)-W(\tau_D)$. Evidently, $\wt W$ is a $\wt{\mathbb F}$-Wiener process, and the newly defined coefficients satisfy (H1)--(H4), (V1), and (V2), moreover, $\wt \alpha(t,x) >0$ \ $\wt \pr$-a.s.\ for all $t\ge 0$, $x\in D$. It is easy to see that $\wt X$ solves
$$
\wt X(t) = \wt X(0) + \int_0^t \Big(a(s,\wt X(s),\omega)ds + \sum_{k=1}^m b_k(s,\wt X(s),\omega)d\wt W_k(s)\Big),\quad t\ge 0.
$$
According to Step 1, there exists $\wt\theta>0$ $\wt \pr$-a.s. such that $\wt X(t)\in D$, $t\in[0,\wt \theta]$ \ $\wt \pr$-a.s. Thus, $X(t)\in D$, $t\in[\tau_D,\tau_D + \wt \theta]$ for $\pr$-a.a.\ $\omega\in\wt\Omega$, which contradicts the definition of $\tau_D$.

\emph{Step 3}. Now we prove the statement in its original form. Let $\set{a^n(t,x),n\ge 1}$ be sequence of coefficients such that (H1)--(H4) holds, for any $T> 0$
$$
\ex{\sup_{t\in[0,T]}\sup_{x,y\in\R}\abs{a^n(t,x)-a(t,x)}^2}\to 0,\quad n\to\infty
$$
and almost surely for any $t\ge 0$, $x\in D$ the assumption of Step 1 is satisfied, i.e.
$$
\alpha^n(t,x) := \left(\varphi'(x), a^n(t,x)\right) + \frac12\sum_{k=1}^m \sum_{i,j=1}^d b_{ki}(t,x)b_{kj}(t,x)\varphi''_{ij}(x)> 0.
$$
One can take, for example, $a_n(t,x) = a(t,x) + n^{-1}\varphi'(x) G(x)$ with a positive smooth function $G\colon \R^d\to R$ which does not vanish on $\partial D$ and decays on infinity sufficiently rapidly so that $\varphi'(x)G(x)$ is bounded together with its derivative.

Let $X^n$ be the solution of
\begin{equation}
X^n(t) = X_0 + \int_0^t \Big(a^n(s,X^n(s))ds + \sum_{k=1}^m b_k(s,X^n(s))dW_k(s)\Big),\quad t\ge 0.
\end{equation}
Then it is well known that $X^n(t)\to X(t)$ in probability locally uniformly in $t$. By Steps 1 and 2,
$X^n(t)\in D$ a.s. for all $t\ge 0$. Since $D$ is closed, we get $X(t)\in D$ a.s.\ for all $t\ge 0$.
\end{proof}
We note that one can extend the results of Theorem~\ref{stochviabilityrandcoeffs} to the case where the constants in (H3), (H4) are random
and 
(H4) holds locally
in
$x,y$.

\section{Comparison theorem for equations with random coefficients}

We formulate the result below as multidimensional, while it is basically about a pathwise comparison in one-dimensional case.

Let $X^q$, $q=1,2$, be solutions of stochastic differential equations
\begin{equation*}
X^q(t) = X^q_0 + \int_0^t \Big(a^q(s,X^q(s))ds + \sum_{k=1}^m b_k(s,X^q_l(s))dW_k(s)\Big),\quad t\ge 0,\
\end{equation*}
where the coefficients $a^q$, $q=1,2$, and $b_k$,  $k=1,\dots,m$, satisfy (H1)--(H4); $X_0^q=(X^q_{01},\dots,X^q_{0d})$, $q=1,2$,  are bounded $\cF_0$-measurable random vectors, and $l\in\set{1,\dots,d}$ is a fixed coordinate number.
\begin{theorem}\label{compartheorem}
Assume that
\begin{enumerate}[\rm (C1)]
\item $X^1_{0l}\le X^2_{0l}$ a.s.,
\item  for any $x^1,x^2\in \R^d$ such that $x^1_l= x^2_l$ it holds
$a^1(t,x^1)\le  a^2(t,x^2)$ a.s.
\end{enumerate}
Then $\pr(X^1_l(t)\le X^2_l(t), t\ge 0)=1$.
\end{theorem}
\begin{proof}
Consider the process $X(t) = (X^1(t),X^2(t))\in\R^{2d}$ and set $\varphi(x) = x^2_l -x^1_l$. Then, in
the notation of Theorem~\ref{stochviabilityrandcoeffs}, $D = \set{X^1_{l}\le X^2_{l}}$,
$\partial D = \set{X^1_{l}= X^2_{l}}$ and $\alpha(t,x) = a^2(t,x)-a^1(t,x)$. By the assumption, $X_0=(X^1_0,X^2_0)\in D$
and $\alpha(t,x)\ge 0$ a.s.\ for $x\in\partial D$. Thus, we get the desired statement from Theorem~\ref{stochviabilityrandcoeffs}.
\end{proof}

\begin{remark}\label{onedimcomparrem}
In \cite{milian}, the following result is given (we reformulate it slightly according to our notation).
\begin{theorem}[\cite{milian}]
Let $X^q$, $q=1,2$, be solutions of stochastic differential equations
\begin{equation*}
X^q(t) = X^q_0 + \int_0^t \Big(a^q(s,X^q(s))ds + \sum_{k=1}^m b^q_k(s,X^q(s))dW_k(s)\Big),\quad t\ge 0,\
\end{equation*}
where $a^q\colon \R_+\times \R^d\to \R^d$ and $b^q\colon \R_+\times \R^d\to \R^d$ satisfy the linear growth and the Lipschitz
continuity assumptions, $X_0^q\in\R$ are non-random. Let $I\subset\set{1,\dots,d}$ be a non-empty set. Then the following
assertions are equivalent:
\begin{enumerate}[\rm (a)]
\item for any $X^q_0=(X^q_{01},\dots,X^q_{0d})$, $q=1,2$, such that $X^1_{0j}\le X^2_{0j}$ for all $j\in I$ it holds
$$
\pr(X^1_j(t)\le X^2_j(t), j\in I,t\ge 0)=1;
$$
\item for any $p\in I$ and $x^i=(x^i_{1},\dots,x^i_{d})$, $i=1,2$, such that $x^1_{j}\le x^2_{j}$, $j\in I$,
$x^1_{p}= x^2_{p}$ it holds
$$
a^1_p(t,x^1) \le a^2_p(t,x^2),\quad b^1_{kp}(t,x^1) = b^2_{kp}(t,x^2), k=1,\dots,m.
$$
\end{enumerate}
\end{theorem}
First we note that it follows from (b) that $b^i_{kp}$, $i=1,2$ are equal and independent of all coordinates
except those from $I$.

Now we argue why comparing more than one coordinate makes only little sense. For example, take $I=\set{1,2}$ and observe that $b^i_{kp}$ are independent of $x$. Indeed,
from (b), for any $p\in I$
\begin{gather*}
b^1_{kp}(t,x_1,x_2,\dots) = b^2_{kp}(t,x_1,x_2\vee x_2',\dots) = b^1_{kp}(t,x_1,x_2\vee x_2',\dots)\\ =
b^2_{kp} (x_1\vee x_2',x_2\vee x_2',\dots) = b^1_{kp} (x_1\vee x_2',x_2\vee x_2',\dots),
\end{gather*}
and similarly
\begin{gather*}
b^1_{kp}(t,x_1,x_2,\dots) =  b^1_{kp} (x_1\vee x_2',x_2\vee x_2',\dots).
\end{gather*}
The dots here can be anything, so we derive $b^1_{kp}(t,x) = b^2_{kp}(t,x) = b_{kp}(t)$. Thus, for $p\in I$, defining
$Y^i(t) = X^i_p(t)- B_p(t)$, where $B_p(t) = \sum_{k=1}^m \int_0^t b_{kp}(s)dW_k(s)$, we can write
the equation for $Y^i$ as
$$
dY^i(t) = a^i(t,Y^i(t)-B_p(t)) =:  f^i(t,Y^i(t))dt.
$$
Though $f^i$ is random, this is an ordinary differential equation. We can solve assuming that all other coordinates and $\omega$
are given. The conditions we have guarantee that $f^1(t,y)\le f^2(t,y)$ disregarding what happens to other coordinates or $\omega$,
so from comparison theorem for ordinary differential equations we conclude that $Y^1(t)\le Y^2(t)$ for all $t\ge 0$ whenever $Y^1(0)=X^1_{0p}\le Y_2(0)= X^2_{0p}$, whence the same is true for $X^i_p(t)$.

Now for $I$ containing a single element we still get that $b(t,x)$ does not depend on other coordinates.
For $a$, we get exactly condition (C2) from Theorem~\ref{compartheorem}. Again, assuming that all other coordinates
of $X^1$ and $X^2$ are given, we get one-dimensional equations with adapted random coefficients. But (C2) guarantees that
this random coefficients are properly ordered whatever happens with other coordinates, so this is essentially a one-dimensional result in a very rigorous sense. It follows that there is nothing new in this result compared to \cite{ikeda-wata} (except randomness of coefficients, but it
can be checked that the proof in \cite{ikeda-wata} works for random coefficients).
\end{remark}

\section{Proof of Theorem \ref{ex-uniq}}
To avoid technical difficulties, we
consider  equation  \eqref{cir-pure} with $\lambda=\frac12$. The proof is similar for remaining cases.

Cases (I), (II).   For any $\varepsilon\in(0,X_0)$ we introduce a smooth coefficient of the form $$d_{\varepsilon}(x)=\sqrt{x}1_{x\geq\varepsilon}+(5/2\varepsilon^{-3/2}x^{2}-3/2\varepsilon^{-5/2}x^3)1_{0\leq x\leq \varepsilon}.$$
The function $d_{\varepsilon}\in C^1(\R)$, and its derivative $$d_{\varepsilon}'(x)=1/2x^{-1/2}1_{x\geq\varepsilon}+(5\varepsilon^{-3/2}x-9/2\varepsilon^{-5/2}x^2)1_{0\leq x\leq \varepsilon}$$ satisfies the Lipschitz condition. Therefore, the following equation with the smooth diffusion
\begin{equation}\label{cir-pure-appr}
dX_{\varepsilon}(t)=aX_{\varepsilon}(t)dt + d_{\varepsilon}(X_{\varepsilon}(t))dB^H(t), t\geq 0, X_{\varepsilon}(0)=X_0>0
\end{equation}
has a unique solution. Let us introduce the Markov moment
$$\tau_{\varepsilon}=\inf\{t>0: X_{\epsilon}(t)\leq\varepsilon\}>0.$$
Obviously, $X_{\varepsilon}(\tau_{\varepsilon})=\varepsilon$ on the set $\{\tau_{\varepsilon}<\infty\}$ and $X_{\varepsilon}(t)> \varepsilon$ for  $t\geq 0$ on the set $\{\tau_{\varepsilon}=\infty\}$. It means that on the interval $[0,\tau_{\varepsilon} )$
there exists the  solution of equation \eqref{cir-pure} because on this interval coefficients of \eqref{cir-pure} and  \eqref{cir-pure-appr} coincide. Moreover, the solution of \eqref{cir-pure} is unique on $[0,\tau_{\varepsilon} )$  because any solution of \eqref{cir-pure} on $[0,\tau_{\varepsilon} )$ can be extended to the solution of \eqref{cir-pure-appr} on $\R_+$. In turn, it means that solution $X$ of equation \eqref{cir-pure}
exists and is unique on the interval $[0, \sup_{\varepsilon>0}\tau_{\varepsilon} )$. We note also that  $\tau_{\varepsilon_1}<\tau_{\varepsilon_2}$
for  $\varepsilon_2<\varepsilon_1$. In order to prove that $X$ is continuous  at the point
$\sup_{\varepsilon>0}\tau_{\varepsilon}$ on
the set $\{\sup_{\varepsilon>0}\tau_{\varepsilon}<\infty\} $ (in this case $X(\sup_{\varepsilon>0}\tau_{\varepsilon})=0$ on
the set $\{\sup_{\varepsilon>0}\tau_{\varepsilon}<\infty\} $), we estimate  the expectation of the solution.

Let $H>2/3$, $\varepsilon>0$,   $t>0$, $R>1$, $L>0$ and $\beta\in(1-H, 1/3)$  be fixed. For  $\delta\in(0,1/3-\beta)$, consider  the process
$$
Y(t)=\int_0^t \frac{{|X(t)-X(u)|}}{(t-u)^{1+2\beta+\delta}}du.
$$
It is well defined on any $[0,\tau_\varepsilon]$ because the numerator does not exceed $C(\omega)(t-u)^{H-\kappa}$ for any $0<\kappa<H$ and we can choose $\kappa$ in such a way that $H-\kappa>2\beta+\delta$.
Define
\begin{gather*}
\nu_L=\inf\{s\geq 0: X(t)+Y(t)\geq L\}\wedge\tau_{\varepsilon},\\
\tilde{\nu}_R=\inf\{s>0:\Lambda_{s, \beta}(B^H)\geq R\}
\end{gather*}
 and $\theta=\theta_{R,L}
=\nu_L\wedge\tilde{\nu}_R.$ According to \eqref{ineq}, we have
\begin{equation}\label{ineq1}
\begin{gathered}
\ex{\sup_{z\leq t}X({z\wedge \theta})}\\
\leq X_0+|a|\int_0^t\ex{X({s\wedge\theta})}ds+\sigma \ex{\sup_{z\leq t}\abs{\int_0^{z\wedge\theta}\sqrt{X(s)}dB^H(s)}}\\
\leq X_0+|a|\int_0^t\ex{X({s\wedge\theta})}ds+C_\beta R\sigma\Big( \int_0^t\frac{\sqrt{\ex{X(s\wedge\theta)}}}{s^\beta}ds\\
{}+\int_0^t\int_0^s\frac{\sqrt{\ex{|X(s\wedge\theta)-X(u\wedge\theta)|}}}
{(s-u)^{1+\beta}}du\,ds\Big).
\end{gathered}
\end{equation}
The integral $I_1=\int_0^t\frac{\sqrt{\ex{X(s\wedge\theta)}}}{s^\beta}ds$ admits an upper  bound
$$I_1\leq C\int_0^t \ex{X(s\wedge\theta)}ds+Ct^{1-2\beta}.$$
To estimate the integral $I_2=\int_0^t\int_0^s\frac{\sqrt{\ex{|X(s\wedge\theta)-X(u\wedge\theta)|}}}
{(s-u)^{1+\beta}}du\,ds$, note that
\begin{equation}\label{ineq1prime}
\begin{gathered}\frac{\sqrt{\ex{|X(s\wedge\theta)-X(u\wedge\theta)|}}}
{(s-u)^{1+\beta}}\\\leq \frac12\Big(\frac{{\ex{|X(s\wedge\theta)-X(u\wedge\theta)|}}}
{(s-u)^{1+2\beta+\delta}}+(s-u)^{-1+\delta}\Big),
\end{gathered}
\end{equation}
whence
$$
I_2\leq C+\int_0^t\int_0^s \frac{{\ex{|X(s\wedge\theta)-X(u\wedge\theta)|}}}
{(s-u)^{1+2\beta+\delta}}du\,ds.
$$
Therefore, denoting
$$
f_1(t)=\ex{\sup_{z\leq t}X(z\wedge\theta)},\;f_2(t)=\ex{Y(t\wedge\theta)}= \int_0^t\frac{{\ex{|X(t\wedge\theta)-X(u\wedge\theta)|}}}
{(s-u)^{1+2\beta+\delta}}du,
$$
we obtain from \eqref{ineq1} and the above estimates that
\begin{equation}\label{ineq2}
\begin{gathered}
f_1(t)\leq C+CR\int_0^tf_1(s)ds+CR\int_0^tf_2(s)ds.
\end{gathered}
\end{equation}
Furthermore, the function $f_2$ admits the following upper bound:
\begin{equation}
\label{ineq3}
\begin{gathered}
f_2(t)\leq \int_0^t(t-u)^{-1-2\beta-\delta}\bigg(C\int_u^tf_1(s)ds+CR\int_u^t\frac{\sqrt{f_1(s)}}{s^\beta}ds\\
+\int_u^t\int_u^s\frac{\sqrt{\ex{|X(s\wedge\theta)-X(r\wedge\theta)|}}}{(s-r)^{1+\beta}}dr\,ds\bigg)du\\ \leq C\int_0^tf_1(s)(t-s)^{-2\beta-\delta}ds+CR\int_0^t \sqrt{f_1(s)}{s^{-\beta}}(t-s)^{-2\beta-\delta}ds\\+
CR\int_0^t \int_0^s\frac{\sqrt{\ex{|X(s\wedge\theta)-X(r\wedge\theta)|}}}{(s-r)^{1+\beta}}dr(t-s)^{-2\beta-\delta}ds.
\end{gathered}
\end{equation}
Consider the second and the third terms in the right-hand side of \eqref{ineq3} separately. Recall that $\delta<1/2-\beta$ and rewrite the
second term as follows:
\begin{equation}\label{ineq4}
\begin{gathered}
\int_0^t \sqrt{f_1(s)}{s^{-\beta}}(t-s)^{-2\beta-\delta}ds
\\=\int_0^t \sqrt{f_1(s)}(t-s)^{1/2-3\beta-3\delta}{s^{-\beta}}(t-s)^{-1/2+2\delta+\beta}ds\\=
\int_0^t {f_1(s)}(t-s)^{1-6\beta-6\delta}ds+\int_0^t{s^{-2\beta}}(t-s)^{-1+2\delta+2\beta}ds\\
=\int_0^t {f_1(s)}(t-s)^{1-6\beta-6\delta}ds+ Ct^{2\delta}.
\end{gathered}
\end{equation}
From the choice of $\beta$ we have $1-6\beta-6\delta>-1.$ To estimate the third term, we apply \eqref{ineq1prime}
and get
\begin{equation}\label{ineq5}
\begin{gathered}
\int_0^t
\int_0^s\frac{\sqrt{\ex{|X(s\wedge\theta)-X(r\wedge\theta)|}}}{(s-r)^{1+\beta}}dr(t-s)^{-2\beta-\delta}ds\\ \leq
\int_0^tf_2(s)(t-s)^{-2\beta-\delta}ds+Ct^{1-2\beta}.
\end{gathered}
\end{equation}
Finally, we get from \eqref{ineq1}--\eqref{ineq5} that
\begin{equation}\label{ineq6}
\begin{gathered}
f_1(t)+f_2(t)\leq CR\Big(1+\int_0^t(f_1(s)ds+f_2(s))(t-s)^{\gamma}ds\Big),
\end{gathered}
\end{equation}
where $\gamma=(1-6\beta-6\delta)\wedge(-2\beta-\delta)>-1$.
It follows from \eqref{ineq6} and the generalized Gronwall inequality that
\begin{equation}\label{ineq7}
\begin{gathered}
f_1(t)+f_2(t)\leq CR\exp\{CRt\}.
\end{gathered}
\end{equation}
Furthermore, the upper bound in \eqref{ineq7} does not depend on $L$. Therefore, by the Fatou lemma,
$$
\ex{\sup_{z\leq t}X(z\wedge\tilde{\nu}_R\wedge\tau_{\varepsilon})}\leq CR\exp\{CRt\},
$$
and by the Lebesgue monotone convergence theorem,
$$
\ex{\sup_{z\leq t}X\big(z\wedge\tilde{\nu}_R\wedge(\sup_{\varepsilon>0} \tau_{\varepsilon})\big)}\leq CR\exp\{CRt\}.
$$
It means that the both components, $\int_0^tX(s)ds$ and $\int_0^t\sqrt{X(s)}dB^H(s)$,
are almost surely bounded and continuous on $\big[0, \tilde{\nu}_R\wedge(\sup_{\eps>0}\tau_{\varepsilon}) \big]$
if $\sup_{\eps>0}\tau_{\varepsilon}<\infty$.
Since $\tilde{\nu}_R\rightarrow\infty$ a.s.,
we claim that $X$ is continuous on $\big[0, \sup_{\eps>0}\tau_{\varepsilon} \big]$ if $\sup_{\eps>0}\tau_{\varepsilon}<\infty$.
But $X(\tau_{\varepsilon})\rightarrow 0$ as $\eps\to0$ if $\sup_{\eps>0}\tau_{\varepsilon}<\infty$,
therefore $X(\sup_{\varepsilon>0}\tau_{\varepsilon})=0$ on the set $\{\sup_{\eps>0}\tau_{\varepsilon}<\infty\}$.

Case (III). It is easy to see that $\nu_0=\sup_{\varepsilon>0}\tau_{\varepsilon}$.
Assume that there exists the solution of equation \eqref{cir-pure} with $\lambda=\frac12$
such that for some $\rho>0$ probability of the event $A=\{\nu_\rho=\inf\{t>\nu_0: X(t=\rho\}<\infty\}$ is nonzero.
Then we can consider the new probability space $\tilde{\Omega}=A$
and repeat the arguments above to prove the uniqueness of the solution of equation \eqref{cir-pure} on some interval
$[\nu_\rho, \nu)$. However, equation \eqref{cir-pure} has on this interval zero solution, and we get a contradiction.


\begin{thebibliography}{xx}

\bibitem{aubin-doss}
Jean-Pierre Aubin and Halim Doss.
\newblock {Characterization of stochastic viability of any nonsmooth set
  involving its generalized contingent curvature.}
\newblock {\em Stochastic Anal. Appl.}, 21(5):955--981, 2003.

\bibitem{BSV} Cristian Bender, Tommi Sottinen and Esko Valkeila.\newblock
Fractional processes as models in stochastic finance.\newblock
In: \emph{Advanced Mathematical Methods for Finance}, pages 75--104. Berlin: Springer, 2011.

\bibitem{Cheridito} Patrick Cheridito.\newblock Regularizing fractional
Brownian motion with a view towards stock price modelling. \newblock PhD
thesis, Zurich, 2001.

\bibitem{ciotir-rascanu}
Ioana Ciotir and Aurel R{\u{a}}{\c{s}}canu.
\newblock Viability for differential equations driven by fractional {B}rownian
  motion.
\newblock {\em J. Differential Equations}, 247(5):1505--1528, 2009.

\bibitem{RamaCont} Rama Cont.\newblock   Long range dependence in financial markets. \newblock
In: \emph{Fractals in Engineering}, pages 159--179. London: Springer,
2005.

\bibitem{CIR} John C. Cox, Jonathan E. Ingersoll  and Stephen
              A. Ross. \newblock  A Theory of the Term Structure of Interest Rates.\newblock
              \emph{Econometrica}, 53(2):385--407, 1985.



\bibitem{doss}
Halim Doss.
\newblock {Liens entre \'equations diff\'erentielles stochastiques et
  ordinaires.}
\newblock {\em Ann. Inst. Henri Poincar\'e, Nouv. S\'er., Sect. B}, 13:99--125,
  1977.

\bibitem{doss-lenglart}
Halim Doss and Eric Lenglart.
\newblock {Sur le comportement asymptotique des solutions d'\'equations
  diff\'erentielles stochastiques.}
\newblock {\em C. R. Acad. Sci., Paris, S\'er. A}, 284:971--974, 1977.

\bibitem{FeyelPradelle}Denis Feyel and Arnaud de la Pradelle. \newblock
The FBM It\^o's formula through analytic continuation. \newblock
\emph{Electronic J. Prob.}, {6}:paper 26, 2001.

\bibitem{filipovic}
Damir Filipovi\'c.
\newblock {Invariant manifolds for weak solutions to stochastic equations.}
\newblock {\em Probab. Theory Relat. Fields}, 118(3):323--341, 2000.

\bibitem{guernual}
Jo{\~a}o Guerra and David Nualart.
\newblock Stochastic differential equations driven by fractional {B}rownian
  motion and standard {B}rownian motion.
\newblock {\em Stoch. Anal. Appl.}, 26(5):1053--1075, 2008.

\bibitem{ikeda-wata}
Nobuyuki Ikeda and Shinzo Watanabe.
\newblock A comparison theorem for solutions of stochastic differential
  equations and its applications.
\newblock {\em Osaka J. Math.}, 14(3):619--633, 1977.

\bibitem{iwbook} 
Nobuyuki Ikeda and Shinzo Watanabe.
\newblock {\em Stochastic differential equations and diffusion processes}. 2nd ed.
\newblock
North-Holland Mathematical Library, 24.
\newblock
 North-Holland, Amsterdam, 1989.

\bibitem{ioffe} Mark Ioffe.\newblock Probability distribution of Cox-Ingersoll-Ross process.\newblock Working Paper, Egar
technology, New York, 2010.




\bibitem{JPS}
 Robert A. Jarrow, Philip Protter  and Hasanjan Sayit. \newblock
 No arbitrage without semimartingales. \newblock
 \emph{The Annals of Applied Probability}
19(2):596–-616, 2009.

\bibitem{melnikov}
Vladislav~Y. Krasin and Alexander~V. Melnikov.
\newblock On comparison theorem and its applications to finance.
\newblock In {\em Optimality and risk---modern trends in mathematical finance},
  pages 171--181. Springer, Berlin, 2009.

\bibitem{kubilius}
K{\c e}stutis Kubilius.
\newblock The existence and uniqueness of the solution of an integral equation
  driven by a {$p$}-semimartingale of special type.
\newblock {\em Stochastic Process. Appl.}, 98(2):289--315, 2002.

\bibitem{milian}
Anna Milian.
\newblock {Stochastic viability and a comparison theorem.}
\newblock {\em Colloq. Math.}, 68(2):297--316, 1995.
\bibitem{mish08} Yuliya S. Mishura. \newblock
\emph{Stochastic calculus for fractional Brownian motion and related processes}. \newblock
Berlin: Springer, 2008.

\bibitem{mbfbm-limit}
Yuliya~S. Mishura and Georgiy~M. Shevchenko.
\newblock Mixed stochastic differential equations with long-range dependence:
  existence, uniqueness and convergence of solutions.
\newblock {\em Comput. Math. Appl.}, to appear, 2012.
\newblock arXiv:1112.2332.
\bibitem{samko}  Stefan G. Samko, Anatoly A. Kilbas and Oleg I.  Marichev.
\newblock \emph{Fractional Integrals and Derivatives. Theory and
Applications.}
\newblock
Gordon and Breach Science Publishers, New York, 1993.

\bibitem{skorokhod}
Anatolii~V. Skorokhod.
\newblock {\em Studies in the theory of random processes}.
\newblock Addison-Wesley Publishing Co., Inc., Reading, Mass., 1965.

\bibitem{Witate} Walter Willinger, Murad S. Taqqu and Vadim Teverovsky. \newblock
 Stock market prices and long-range dependence. \newblock
\emph{Finance and Stochastics}, 3(1):1--13, 1999.


\bibitem{yamada}
Toshio Yamada.
\newblock On a comparison theorem for solutions of stochastic differential
  equations and its applications.
\newblock {\em J. Math. Kyoto Univ.}, 13:497--512, 1973.

\bibitem{zabczyk}
Jerzy Zabczyk.
\newblock {Stochastic invariance and consistency of financial models.}
\newblock {\em Atti Accad. Naz. Lincei, Cl. Sci. Fis. Mat. Nat., IX. Ser.,
  Rend. Lincei, Mat. Appl.}, 11(2):67--80, 2000.
  \bibitem{Zahle} Martina Z\"{a}hle.
 \newblock On the link between fractional and
stochastic calculus. \newblock
 In: Grauel, H., Gundlach, M. (eds.) \emph{Stochastic
Dynamics}, pages 305--325. New York: Springer,  1999.
  \end{thebibliography}
\end{document}